\documentclass[11pt,a4paper]{amsart}

\usepackage{amsthm, amsfonts, amssymb, amsmath, latexsym, enumerate,array}
\usepackage{amscd} 
\usepackage[all]{xy}
\usepackage{pstricks,graphicx}
\usepackage{stmaryrd}
\usepackage[utf8]{inputenc}
\usepackage{hyperref,cite,mdwlist,mathrsfs}
\usepackage{tikz}
\usepackage{graphicx}
\usepackage{pgf,tikz}
\usetikzlibrary{arrows}

\usepackage{anysize}
\marginsize{3.0cm}{3.0cm}{3.3cm}{3.3cm}
\usepackage{setspace}

\newtheorem{thm}{Theorem}[section]
\newtheorem{lem}[thm]{Lemma}

\newtheorem{prop}[thm]{Proposition}

\newtheorem*{thm*}{Theorem}
\newtheorem*{cnj*}{Conjecture}

\theoremstyle{definition}
\newtheorem{rmk}[thm]{Remark}
\newtheorem{eg}[thm]{Example}
\newtheorem{dfn}[thm]{Definition}
\newtheorem*{remark}{Remark}

\newtheorem*{conj*}{Conjecture}

\newcommand {\PP}{\mathbb{P}}

\newcommand{\cA}{\mathcal{A}}
\newcommand{\cB}{\mathcal{B}}

\newcommand{\cT}{\mathcal{T}}

\newcommand{\cZ}{\mathcal{Z}}

\newcommand{\cO}{\mathcal{O}}

\newcommand{\m}{\mathfrak m}

\DeclareMathOperator{\Der}{\mathrm{Der}}
\DeclareMathOperator{\Proj}{Proj}

\DeclareMathOperator{\HH}{H}

\newcommand{\C}{\mathbb C}

\newcommand{\Q}{\mathbb Q}

\newcommand{\p}{\mathbb P}

\newcommand{\kk}{\mathbb{K}}

\allowdisplaybreaks[1]
\begin{document}


\title{Free curves, Eigenschemes and Pencils of curves}

\author[R.Di Gennaro]{Roberta Di Gennaro}
\address{Di Gennaro: Dipartimento di Matematica e Applicazioni “Renato Caccioppoli” \\
 Universit\`a degli Studi di Napoli Federico II\\
  80126 Napoli, Italy}
\email{digennar@unina.it}

\author[G. Ilardi]{Giovanna Ilardi}
\address{Ilardi: Dipartimento di Matematica e Applicazioni “Renato  Caccioppoli” \\
  Universit\`a degli Studi di Napoli Federico II\\
  80126 Napoli, Italy}
\email{giovanna.ilardi@unina.it}

\author[R.M. Mir\'o-Roig]{Rosa Maria Mir\'o-Roig}
\address{Mir\'o-Roig: Department de Mathem\`atiques i Inform\`atica\\
  Universitat de Barcelona\\
  08007 Barcelona, Spain}
\email{miro@ub.edu}

\author[H. Schenck]{Hal Schenck}
\address{Schenck: Mathematics Department \\
  Auburn University \\
   Auburn AL 36849 USA}
\email{hks0015@auburn.edu}

\author[J. Vall\`es]{Jean Vall\`es}
\address{Vall\`es: Department of Mathematics \\
  Universit\'e de Pau et des Pays de l'Adour \\
  64012 PAU Cedex - France}
\email{jean.valles@univ-pau.fr}


\thanks{Mir\'o-Roig has been partially supported by the grant PID2020-113674GB-I00. Schenck was  supported by NSF 2006410.  Di Gennaro and Ilardi were partially supported by GNSAGA-INDAM}

\begin{abstract} Let $R=\kk[x,y,z]$. A reduced plane curve $C=V(f)\subset \PP^2$ is {\em free} if its associated module of tangent derivations $\mathrm{Der}(f)$ is a free $R$-module, or equivalently if the corresponding sheaf $T_ {\PP^2
}(-\log C)$ of vector fields tangent to $C$ splits as a direct sum of line bundles on $\p^2$. In general, free curves are difficult to find, and in this note, we describe a  new method for constructing free curves in $\PP^2$. The key tools in our approach are eigenschemes and pencils of curves, combined with an interpretation of Saito's criterion in this context. Previous constructions typically applied only to curves with quasihomogeneous singularities, which is not necessary in our approach. We illustrate our method by constructing large families of free curves. 
\end{abstract}
\maketitle
\section{Introduction}
In this paper, we study the module (or sheaf) of derivations tangent to a reduced curve $C \subseteq \p^2$. This is a classical topic of algebraic geometry, and fits into the following broader picture: let $D$ be a divisor on a smooth complex variety $X$. When $D$ is a
normal crossing divisor, Deligne \cite{d}  constructed a mixed Hodge
structure on $U = X \setminus D$ using the logarithmic de Rham complex
$\Omega_X^\bullet(-\mbox{log }D)$. Building on this, in \cite{S} 
Saito defined the sheaf of derivations tangent to $D$ and (dually) the sheaf of logarithmic 
one-forms with pole along $D$.
\begin{dfn}\label{D1}The module of tangent derivations is a sheaf of
$\mathcal{O}_X$--modules, such that if $f \in \mathcal{O}_{X,p}$ is a local
defining equation for $D$ at $p$, then 
\[
(\cT_X(-\log D))_p = \{ \theta \in \cT_X \, | \, \theta(f) \in \langle f \rangle \}.
\]
\end{dfn}
\noindent When $D$ is a normal-crossing divisor, $\cT_X(-\log D)$ is always locally free. Saito shows that if $X$ has dimension $n$ then $\cT_X(-\log D)$ is a locally free sheaf if and only if locally there exist $n$ derivations
\begin{equation}\label{saitoCriterion}
\theta_i = \sum\limits_{j=1}^n f_{ij}\frac{\partial}{\partial x_j} \in \cT_X(-\log  D)_p
\end{equation}
such that the determinant of the matrix $[f_{ij}]$ of coefficients of the derivations $\{\theta_1, \ldots, \theta_n\}$ above is a unit multiple of the local defining equation for $D$.
 See \cite{gms}, \cite{liaoSchulze}, \cite{mondSchulze}, \cite{STY}, \cite{V} for recent work on $\cT_X(-\log D)$.  Even when the ambient space $X$ is a projective space, determining if the module of tangent derivations is locally free is non-trivial. 
 
The module of derivations tangent to $D$ is a reflexive sheaf. So, since a reflexive sheaf on a surface is always locally free, when $X=\p^2$ the module of derivations is locally free. Our interest is when it splits as $\cO_{\p^2}(a) \oplus \cO_{\p^2}(b)$, in which case the divisor (now a curve $C$) is said to be a {\em free} curve with exponents $(a,b)$. In general, free curves are difficult to find, and the point of the present work is to describe a new method to construct free curves based on the theory of eigenschemes of tensors. 

\subsection{Algebraic preliminaries}
Let $\kk$ be an algebraically closed field of characteristic $0$ and  $R=\oplus_{k\ge 0}R_k=\kk[x_0,\ldots ,x_n]$ be the graded ring in $n+1$ variables with $\PP^n=\Proj(R)$.
\begin{dfn}
For $R$ as above, the module of $\kk$-derivations $\mathrm{Der}_{\kk}(R)$ is free of rank $n+1$, with basis 
$\{ \partial_{x_0}, \ldots ,\partial_{x_n}\}$. For a reduced homogeneous polynomial $f\in R_{d \ge 1}$, the module of derivations 
$\mathrm{Der}(f)$ tangent to $V(f)$ is defined as
$$\mathrm{Der}(f):=\{\delta \in \mathrm{Der}_{\kk}(R) \mid \delta(f)\in \langle f \rangle\}.$$
The divisor $V(f)$ is {\em free} if $\mathrm{Der}(f)$ is a free $R$-module. The Euler derivation 
$$\delta_E=\sum_{i=0}^n x_i\frac{\partial}{\partial x_i}$$
satisfies $\delta_E(f)=d  f$. Hence, for any $\delta \in \mathrm{Der}(f)$ we have the decomposition
$$
\delta=\delta'+\frac{1}{d}\frac{\delta(f)}{f}\delta_E, \mbox{ with }\delta'=\delta - \frac{1}{d}\frac{\delta(f)}{f}\delta_E\mbox{ and }\delta'(f)=0,
$$
which yields the decomposition 
$$
\mathrm{Der}(f)=R\delta_E\oplus \mathrm{Der}_0(f), \mbox{ where } \mathrm{Der}_0(f)=\{\delta \in \mathrm{Der}_{\kk}(R) \mid \delta(f)=0\}.$$
\end{dfn}
\noindent Let $\nabla(f)=(\partial_{x_0} f, \ldots, \partial_{x_n} f)$ be the vector of partial derivatives. Then $\mathrm{Der}_0(f)$ is simply the kernel of the Jacobian map 
$$ \begin{CD} R^{n+1} @>\nabla(f)>> R(d-1).
\end{CD}$$
\begin{eg}\label{TeraoFree}The study of {\em hyperplane arrangements} focuses on the case where the divisor is a (reduced) union of hyperplanes $H_i$ in $\p^n$; by convention in this case the divisor is written as $\cA = \cup H_i$. In \cite{OS} Orlik-Solomon showed that the cohomology ring
$H^*(U_{\cA},\Q)$ of the affine arrangement complement $U_{\cA} \subseteq \C^{n+1}$ has
a purely combinatorial description, and a renowned theorem of Terao \cite{T} relates $H^*(U_{\cA},\Q)$ to the freeness of $V(f)$: 
\vskip .1in
\noindent {\bf Theorem (Terao)}: For a reduced hyperplane arrangement $\cA \subseteq \p^n$ with $\cA = V(f)$, if $Der(f) \simeq \oplus_{i=1}^{n+1} R(-a_i),$ then the Poincar\'e polynomial of $H^*(U_{\cA},\Q)$ satisfies
 \[
 P(H^*(U_{\cA},\Q),t) = \prod\limits_{i=1}^{n+1} (1+a_it).
\]
\end{eg}
The condition that $D = V(f)$ is a free divisor on $\p^n$ is
equivalent to the Jacobian ideal $J_f$ of $f$ generated by $\nabla(f)$ being arithmetically Cohen-Macaulay
of codimension two. Such ideals are completely described by the
Hilbert-Burch theorem \cite{e}: if $I = \langle g_1,\ldots,
g_m\rangle$ is Cohen-Macaulay of codimension two, then $I$ is defined
by the maximal minors of the $m \times (m-1)$ matrix of the first
syzygies of the ideal $I$.

Combining this with Euler's formula for a homogeneous
polynomial shows that a free divisor $V(f)$ on $\p^n$ has a very
constrained structure: $f = \det(M)$ for an $(n+1) \times (n+1)$ matrix $M$, with one column consisting of the variables, and the
remaining $n$ columns the minimal first syzygies on $\nabla(f)$. In particular, $V(f)$ is a special type of determinantal hypersurface. As shown by Beauville \cite{B}, smooth determinantal hypersurfaces are quite rare. We focus here on singular curves in $\p^2$ which are reduced but not irreducible.

\subsection{History and Techniques for Free Divisors}
Following the early work of Deligne and Saito, the appearance of Terao's freeness theorem led to much subsequent work on freeness for hyperplane arrangements. In this setting, there is a natural inductive approach introduced by Terao in \cite{t2} to the study of freeness, involving the interplay between adding a hyperplane $H$ to $\cA$, and restricting $\cA$ to $H$. This inductive approach was generalized to the case of rational plane curve arrangements in \cite{cmh2}, and to arrangements of higher genus plane curves in \cite{STY}. 

The results in \cite{STY} and \cite{cmh2} require that the curve $C = \cup C_i$ has {\em quasihomogeneous} singularities, which means that at each singular point, the Milnor number and Tjurina number are equal. This is a subtle property which is nontrivial to verify. An important feature of our work in this paper is that no assumptions are made (or needed) about the type of singularity. Another technique for studying free curves in $\p^2$ appears in work \cite{V} of Vall\`es, and our approach builds on \cite{V}. First, we recall some terminology and background.
\begin{dfn}
 A linear system is a subspace $V \subseteq H^0(\cO_{\p^n}(d))=R_d$. If the dimension of $\p(V)$ is one or two, the corresponding linear systems are called {\em pencils} or {\em nets}.
 \end{dfn}

\begin{thm}\label{JeanTheorem} (Vall\`es,\cite{V})
Let $C = \cup C_i \subseteq \p^2$ be a union of curves $C_i$ from a pencil $P$ of curves of  degree $d$, such that $P$ has a smooth base locus and let $f$ be the corresponding reduced homogeneous polynomial. Then $C=V(f)$ is a free divisor with exponents $(2d-2, N-2d+1)$ where $N=\mathrm{deg}(C)$ if and only if $C$ contains all the singular members of the pencil and $\nabla(f)$ is a local complete intersection.
\end{thm}

One of the key tools in \cite{V} is the construction of a canonical syzygy on $\nabla(f)$, with $f$ as above. Let $G_1$ and $G_2$ be two general elements of the pencil. We show in Lemma \ref{deriv-can} that the $2\times 2$ minors $\nabla(G_1)\wedge \nabla (G_2)$ of the matrix 
\[
\left[ \!  
\begin{array}{ccc}
\frac{\partial G_1}{\partial x_1} & \frac{\partial G_1}{\partial x_2} &\frac{\partial G_1}{\partial x_3}\\
\frac{\partial G_2}{\partial x_1} & \frac{\partial G_2}{\partial x_2} &\frac{\partial G_2}{\partial x_3}
\end{array}\! \right]
\]
are a syzygy on $\nabla(f)$ .
We can use this method to study more general pencils, but Theorem \ref{JeanTheorem} does not apply when the base locus is not smooth. Even for a pencil of smooth conics sharing a tangent line the Jacobian ideal may not be a local complete intersection: we show in Example \ref{conics} that the Tjurina and Milnor numbers can differ at a singular point. The total Tjurina number (sum of the Tjurina numbers at singular points) determines the second Chern class of the module of tangent derivations, 
and is quite subtle. We discuss this more in \S \ref{conics}. 

Another interesting case which is not covered by Theorem \ref{JeanTheorem} is when the pencil is generated by two multiple curves, such as the pencil $(f^3,g^2)$ where $V(f)$ is a smooth conic and $V(g)$ a smooth cubic. All curves of the pencil are singular along the base locus $V(f)\cap V(g).$ This example was introduced by Zariski in \cite{Za1} and \cite{Za2}:  he constructed  two sextic curves $C_1$ and $C_2$, each with six ordinary cusps, such that the complements $\p^2 \setminus C_1$ and $\p^2 \setminus C_2$ are not homeomorphic. The difference between the two is that $C_1$ has all cusps on a smooth conic, and $C_2$ does not. As a consequence, the fundamental groups of the complements are different. In Example \ref{sextic-pencil}, 
we describe in detail the case of a triangle $V(g)$ meeting a smooth conic $V(f)$ along six distinct points; when $V(g)$ is smooth an analysis appears in \cite{Va-Ng}. For another perspective on freeness of curves in a pencil, see Dimca \cite{di}, and for a close variant to freeness, see Abe \cite{abe}.
\vskip .05in
\noindent{\bf Acknowledgements.} 
Our project began at the 2019 CIRM workshop ``Lefschetz properties in algebra, geometry, and combinatorics''
and we thank CIRM for a great workshop.  
\vskip -.01in
 \section{First Key Tool: Eigenschemes}
\noindent We begin with the definition of an eigenscheme in our context. 
 \begin{dfn}
  The eigenscheme associated to 
 three homogeneous polynomials $(P_1,P_2,P_3)$ of the same degree $n\ge 1$ 
 is the closed subscheme  $\Gamma \subset \p^2$ defined by the $2\times 2$ minors of the matrix 
 $$M=\begin{pmatrix}   x & P_1  \\
y  &P_2 \\
  z & P_3
   \end{pmatrix}.$$
\end{dfn} 
 
  When $V(P_1,P_2,P_3)$ contains a curve this curve is also in the eigenscheme; note that even if $V(P_1,P_2,P_3)$ is a finite set of points 
or empty, the eigenscheme may be of codimension one: for example given $f, g,Q_1,Q_2,Q_3$ homogeneous forms such that \[
P_1=xf+gQ_1, P_2=yf+gQ_2, P_3=zf+gQ_3,
\] 
then the eigenscheme of $(P_1,P_2,P_3)$ clearly contains $V(g)$.

\par Assume now that the eigenscheme $\Gamma$ associated to $(P_1,P_2,P_3)$ is a finite scheme. Then it is defined by
$$\begin{CD}
    0@>>> \cO_{\p^2}(-1)\oplus \cO_{\p^2}(-n)  @>M>> 
      \cO_{\p^2}^3 @>>>  \mathcal{J}_{\Gamma}(n+1) @>>>0,
    \end{CD}$$ 
    where $\mathcal{J}_{\Gamma}$ is its ideal sheaf.  Its  length is 
  $$c_2(\mathcal{J}_{\Gamma}(n+1))=1+n+n^2.$$
   Set theoretically, one sees easily that $\Gamma$ consists of the union of the indeterminacy locus and the fixed points of the rational map
   $$\p^2 ---\rightarrow \p^2, \,\, p\mapsto (P_1(p),P_2(p),P_3(p)).$$
  Moreover, it is locally a complete intersection. Indeed, it is the zero locus of a suitable section of a rank two vector bundle on $\PP^2$; namely of a twist of the tangent bundle $T_{\PP^2}$ of $\PP^2$ as it  can be seen using the following commutative diagram
 
   $$\begin{CD}
     @. 0 @. 0  @.\\
     @. @VVV @VVV @.\\
   0@>>> \cO_{\p^2}(-1) @= \cO_{\p^2}(-1) @. \\
   @. @VVV @VVV @.\\
     0@>>> \cO_{\p^2}(-1)\oplus \cO_{\p^2}(-n)  @>M>> 
      \cO_{\p^2}^3 @>>>  \mathcal{J}_{\Gamma}(n+1) @>>>0\\
      @. @VVV @VVV @| \\
     0 @>>>  \cO_{\p^2}(-n)@>>> T_{\p^2}(-1) @>>>  \mathcal{J}_{\Gamma}(n+1) @>>>0 \\
      @. @VVV @VVV @.\\
       @. 0 @. 0  @.
    \end{CD}$$ 
\noindent For more details about  eigenschemes we refer to \cite{ASS} and \cite{BGV}.
 \newpage
\begin{lem} \label{contain-gamma1} Assume that $\Gamma$ is a finite scheme. Let $f\in R_N$ with 
$N\ge n+1$. Then  
$f\in \HH^0(\mathcal{J}_{\Gamma}(N))$ if and only if there exists \[
(Q_1,Q_2,Q_3)\in R_{N-(n+1)}^3\] such that 
$$\mathrm{det} 
\begin{pmatrix}  x & P_1  &Q_1  \\
y & P_2  &Q_2 \\
  z & P_3  &Q_3
   \end{pmatrix}= cf, \,\, c\in \kk^*$$
\end{lem}
\begin{proof} Let us denote by $I_{\Gamma}$  the saturated ideal 
$\oplus_{j} \HH^0(\mathcal{J}_{\Gamma}(j))$.
 Since $R_{N-(n+1)}^3 \rightarrow( I_{\Gamma})_{N}$ is surjective (because $I_{\Gamma}$ is saturated), a curve $V(f)\subset \PP^2$ of degree $N$ containing $\Gamma$ has an equation of the type 
\[
f=Q_1\m_1+Q_2\m_2+Q_3\m_3=0,
\] where $\bigwedge^2 M=(\m_1,\m_2,\m_3)$. This is clearly equivalent to 
$$f=\mathrm{det} 
\begin{pmatrix}  x & P_1  &Q_1  \\
y & P_2  &Q_2 \\
  z & P_3  &Q_3
   \end{pmatrix}= Q_1\m_1+Q_2\m_2+Q_3\m_3.$$
   The converse is immediate.
\end{proof}
Now let us show how this eigenscheme is related to the notion of freeness for curves.  Recall that we say a reduced polynomial $f$ (or the reduced curve $V(f)\subset \PP^2$) is free if and only if $\mathrm{Der}_0(f)$ (or equivalently $\mathrm{Der}(f)$) is a free $R$-module. So $f$ is free if and only if $\mathrm{Der}_0(f)=R(-a)\oplus R(-b)$ with $0\le a\le b$  and $a+b+1=\mathrm{deg}(f)$, or equivalently 
$\mathrm{Der}(f)=R(-1)\oplus R(-a)\oplus R(-b)$. In this situation, we will use the terminology ``$f$ is free with exponents $(a,b)$''.

\smallskip

Let $C=V(f)$ be a reduced plane curve of degree $d$. 
Let us recall that, according to Saito's criterion \cite{S}, the curve $C$ is free with exponents $(a,b)$
if and only if there exist two derivations 
$$\delta= 
P_1\partial_x+P_2\partial_y+P_3\partial_z \,\, \mathrm{and} \,\, \mu=Q_1\partial_x+Q_2\partial_y+Q_3\partial_z$$ of degree $a$ and $b$ belonging to $\mathrm{Der}(f)$ such that $$
\mathrm{det}\begin{pmatrix}  x&P_1&Q_1   \\
 y&P_2&Q_2  \\
 z&P_3&Q_3
   \end{pmatrix}=cf, $$ where $c\in \kk^*$.
   
\smallskip

Let us introduce now the kernel of a derivation.
\begin{dfn}
Let 
\[\delta=P_1\partial_x+P_2\partial_y+P_3\partial_z
\] be a non zero irreducible derivation of degree $a\ge 1$. The graded module of homogeneous polynomials with $\delta$ as a tangent derivation: $$K(\delta)=\oplus_{d\ge 0} K(\delta)_d :=\{f\in R\,|\, \delta (f)\in (f)\}$$ is called  \underline{kernel} of the derivation $\delta$.
\end{dfn}
\begin{rmk}
$F\in K(\delta) \Leftrightarrow \delta \in \mathrm{Der}(F).$
\end{rmk}
As an immediate consequence of Lemma \ref{contain-gamma1} we obtain:\begin{thm} \label{de-de-new}
 Let $\delta=P_1\partial_x+P_2\partial_y+P_3\partial_z$ be a non zero irreducible derivation of degree $a\ge 1$ such that its eigenscheme $\Gamma_{\delta}$ is a finite scheme.
Let $f$ be a reduced homogeneous polynomial such that  $ f\in K(\delta)_{d\ge a+1}$, then
  the plane curve
  $V(f)$ is free with exponents $(a,d-a-1)$ if and only if $f\in I_{\Gamma_{\delta}}.$
 \end{thm}

 \begin{proof}
 Let us assume first that $V(f)$ is free with exponents $(a, d-a-1)$, so that $\delta$ can be chosen as a generator of  $\mathrm{Der}(f)$. Let us denote by $\mu=Q_1\partial_x+Q_2\partial_y+Q_3\partial_z$ a generator of degree $d-1-a$. Then by Saito's criterion we have 
  $$
\mathrm{det}\begin{pmatrix}  x&P_1&Q_1   \\
 y&P_2&Q_2  \\
 z&P_3&Q_3
   \end{pmatrix}=cf, \quad c\in \kk^*, $$  which proves directly that $f\in I_{\Gamma_{\delta}}.$
   
 \smallskip
   
 Reciprocally, assume that $f\in I_{\Gamma_{\delta}}.$ Then by Lemma \ref{contain-gamma1} there exist three polynomials 
 $(Q_1,Q_2,Q_3)\in R_{d-a-1}^3$
  such that 
     $$\mathrm{det} 
\begin{pmatrix}  x & P_1  &Q_1  \\
y & P_2  &Q_2 \\
  z & P_3  &Q_3
   \end{pmatrix}= f.$$ To conclude with Saito's criterion  it just remains to verify that the corresponding derivation 
$\mu=Q_1\partial_x+Q_2\partial_y+Q_3\partial_z$ verifies   
   $\mu \in \mathrm{Der}(f)$.
   
\smallskip

Let  $M=\begin{pmatrix}  x & P_1  &Q_1  \\
y & P_2  &Q_2 \\
  z & P_3  &Q_3
   \end{pmatrix}$ and $Co(M)^{T}$ be the transpose matrix of its cofactors.
   Let us denote by $\m_x,\m_y,\m_z, \m_{P_i}$ and $\m_{Q_i}$ the cofactors of $x,y,z, P_i$ and $Q_i$ respectively. 
   Recall that $$ M\, Co(M)^{T}=Co(M)\, M^{T}=f I.$$
   Multiplying by $\nabla(f)$, as $x\partial_xf+y\partial_y f+z\partial_zf=df$ and $\delta(f)=P_1f_x+P_2f_y+P_3f_z=Kf$ for some homogeneous polynomial $K $ we get  
   $$Co(M)\, M^{T}\, \nabla(f)=Co(M)\, \begin{pmatrix}  df   \\
 Kf  \\
 \mu(f)
   \end{pmatrix}=f \nabla(f).$$
   This gives the following system of equations: 
   $$
   \left\{
   \begin{array}{ccc}
    \m_{Q_1} \,\mu(f)&=&f (\partial_xf-d\m_x-K\m_{P_1})\\
                 \m_{Q_2}\, \mu(f)&=&f (\partial_yf-d\m_y-K\m_{P_2})\\
                 \m_{Q_3}\, \mu(f)&=&f (\partial_zf-d\m_z-K\m_{P_3})
                 \end{array}
                 \right.
                $$
  If $\mu \notin \mathrm{Der}(f)$ then $f$ does not divide $\mu(f)$ meaning that there is a irreducible factor $g$ of $f$ which is not an irreducible factor of $\mu(f)$. By the Gauss lemma, this factor $g$ divides $\m_{Q_i}$ for $i=1,2,3$. But the $\m_{Q_i}$ are the three generators of the ideal $I_{\Gamma_{\delta}}$ which contradicts the finiteness of  $\Gamma_{\delta}$.             
                Then $\mu \in \mathrm{Der}(f)$.                 \end{proof}
\begin{rmk}
Note that fixing the exponents is necessary for Theorem~\ref{de-de-new} to hold. To see this, consider a free curve of degree $5$ with exponents $(2,2)$. Let $\mu$ and $\nu$ two derivations of degree $2$ generating its logarithmic module. Let $\delta=x\mu+y\nu$ a derivation of degree $3$. Then $f\in K(\delta)_5$ is free but its exponents are not $(3,1)$ and as a consequence $f\notin I_{\Gamma_{\delta}}$. 

In addition, if $f\in K(\delta)_{d}$ with $d\le a$ then $f$ can be free but $\delta$ will not be a generator of its associated logarithmic module, for degree reasons.
\end{rmk}

\section{Second Key Tool: Pencils of Curves}
As  well known examples of free arrangements we have:
\begin{enumerate}
\item The Ceva-Braid arrangement 
\[xyz(x-y)(x-z)(y-z)=0\] is free with exponents $(2,3)$.
\item The Hesse arrangement \[\displaystyle \prod_{\epsilon =\infty, 1,j,j^2}(x^3+y^3+z^3-3\epsilon xyz)=0\] is free with exponents  $(4,7)$.
\item The Fermat arrangement \[(x^n-y^n)(x^n-z^n)(y^n-z^n)=0\] is free with exponents $(n+1,2n-2)$.
\end{enumerate}
A main observation pointed out to the fourth author by Artal and Cogolludo is that each of the three divisors above is the union of all the singular members of the pencil of 
\begin{enumerate}
\item conics $[(x-y)z,y(x-z)]$ for the Ceva-Braid arrangement.
\item cubics $[x^3+y^3+z^3, xyz]$ for the Hesse arrangement.
\item $n$-ics  $[x^n-y^n,x^n-z^n]$ for the Fermat arrangement.
\end{enumerate}
\begin{rmk}
Artal and Cogolludo suggested  that this phenomenon should hold for any pencil. 
When the base locus of the pencil is smooth and the singular curves of the pencil have only quasihomogeneous singularities then this is indeed true, and is proved in \cite{V}. The proof relies on the existence of a canonical derivation
$\delta_{f,g}$ associated to a pencil $[f,g]$.
The obstruction to proving the result with non-smooth base locus is the lack of control over the nature and numerical contributions of the singular points. 
\vskip .1in

\noindent The main idea in this paper is that rather than considering the base locus of the pencil and the singular points  of the singular curves of the canonical derivation $\delta_{f,g}$ as in \cite{V}, we consider the eigenscheme of $\delta_{f,g}$.
\end{rmk}
\subsection{Canonical derivation associated to a pencil}

We now consider two reduced polynomials $f\in R_n$ and $g\in R_m$ with no common factor. Define a derivation as follows:
$$\delta_{f,g}:=[\nabla f\wedge \nabla g].\nabla=\mathrm{det} 
\begin{pmatrix}   \partial_x f& \partial_x g  &\partial_x  \\
\partial_y f& \partial_y g  &\partial_y \\
  \partial_z f& \partial_z g  &\partial_z
   \end{pmatrix}.$$
Then, we have 
\begin{lem} \label{deriv-can} Let $F_k=0$ be the union of $k\ge 1$ curves in the pencil generated 
by $(f^a,g^b)$ where $\mathrm{lcm}(n,m)=a\times n=b\times m$.
The derivation $\delta_{f,g}$ associated to the pair $(f,g)$,  verifies:
\begin{enumerate}
\item $ \delta_{f,g}\in \mathrm{Der}_0(f)\cap \mathrm{Der}_0(g)$,
\item $\delta_{f,g}\in \mathrm{Der}_0(F_k)$,
\item If $F_k=FG$ where $F$ and $G$ are two polynomials with no common factor then $\delta_{f,g}\in \mathrm{Der}(F)\cap \mathrm{Der}(G)$.
\end{enumerate}
\end{lem}
\begin{proof}
The first assertion is an immediate consequence of the definition of $\delta_{f,g}$.
The Leibniz rules for the derivation of a product and a power imply the second assertion.
Let $F_k=FG$. By $(2)$ and the Leibniz rule we have 
$$0=\delta_{f,g}(F_k)=G\delta_{f,g}(F)+F\delta_{f,g}(G),$$
that is $$G\delta_{f,g}(F)=-F\delta_{f,g}(G).$$
Since, by hypothesis $F$ and $G$ do not share any factor, this implies that $F \,|\, \delta_{f,g}(F)$ and $G \,|\, \delta_{f,g}(G)$ proving  that 
$\delta_{f,g}\in \mathrm{Der}(F)\cap \mathrm{Der}(G)$.
\end{proof}

\subsection{Eigenscheme associated to a pencil of curves}
With the same hypothesis as above, we consider the rational map 
$$ \phi_{f,g}: \p^2 \longrightarrow \p^2, \,\, p \mapsto [\nabla f\wedge \nabla g](p)$$
induced by the derivation $ \delta_{f,g}.$ We recall that the eigenscheme $\Gamma$ associated to 
$ \delta_{f,g}$ consists set theoretically of the union of the indeterminacy locus  of $\phi_{f,g}$  and the set of fixed points of $\phi_{f,g}$, i.e. those points $p$ such that $\phi_{f,g}(p)=p$. 
Since 
$p$ and $ [\nabla f\wedge \nabla g](p)$ are the same projective point, the point $p$ is orthogonal to both $\nabla f(p)$ and $\nabla g(p)$ ; this implies by Euler's formula for homogeneous polynomial that $f(p)=g(p)=0$ i.e. $p\in \cB=V(f)\cap V(g)$. The set of fixed points is finite but this is not always the case for the indeterminacy locus $V(\nabla f\wedge \nabla g)$  of $\phi_{f,g}$. 
\begin{rmk}
When $n=m$, 
 $V(\nabla f\wedge \nabla g)$ contains a curve if and only if the pencil contains  non reduced curve, say $f+g=u_1^{r_1}\cdots u_t^{r_t}$ where 
 $u_i=0$ are the reduced and irreducible factors of $f+g$ and at least one $r_i$ is greater than $2$. In this case replacing $f$ by $f+g$ in the pencil, one sees that the derivation $\delta_{f,g}$ is a multiple of  $\delta_{h,g}$ where $h=u_1\cdots u_t$. We then study the eigenscheme associated to $\delta_{h,g}$ (see for instance Example \ref{eigen-conics}).
\end{rmk}
Let us describe more precisely the eigenscheme associated to $\delta_{f,g}$.
\begin{prop}\label{key} Let $f\in R_n$ and $g\in R_m$ be two reduced polynomials without common factors such that $V(\nabla f\wedge \nabla g)$  is a finite scheme. Let $a$ and $b$ integers such that $\mathrm{lcm}(n,m)=a\times n=b\times m$. Let $\Gamma$ be the eigenscheme associated to $\delta_{f,g}.$
\begin{enumerate}
\item The number of  curves,  different from $V(g^{b})$ and $V(f^{a})$   in the pencil $\mathcal{C}=(f^{a},g^{b})$  that are singular outside the base locus $\cB=V(f)\cap V(g)$ is finite and bounded by 
$(n-1)^2+(n-1)(m-1)+(m-1)^2$ which is the length of the scheme 
$\cZ= V(\nabla f \wedge \nabla g)$.
\item The scheme $\Gamma$ is the union of the schemes $\cB$ and $\cZ$. 
\end{enumerate}
\end{prop}
\begin{proof}
To prove $1)$, assume that $\lambda f^a+\mu g^b$ is singular at $p\notin \cB$. Then 
$\nabla (\lambda f^a+\mu g^b)(p)=0$. This gives 
$$ \nabla (\lambda f^a+\mu g^b)(p)=a\lambda f(p)^{a-1}\nabla f(p)+
b\mu g(p)^{b-1}\nabla g(p)=0,$$
with by hypothesis $f(p)\neq 0$ or $g(p)\neq 0$. So the above equation is a relation between $\nabla f(p)$ and $\nabla g(p)$ proving that $p\in \cZ$. This last scheme is defined by 
$$\begin{CD}
     0@>>> \cO_{\p^2}(1-n)\oplus \cO_{\p^2}(1-m)  @>(\nabla f, \nabla g)>> 
      \cO_{\p^2}^3 @>>>  {\mathcal{J}}_{\cZ}(n+m-2) @>>>0
    \end{CD}$$
from which it follows that it has length $(n-1)^2+(n-1)(m-1)+(m-1)^2$.
 
 \smallskip
 
For $2)$, we consider the three schemes, $\cB=V(f)\cap V(g)$, $\cZ=V(\nabla f \wedge \nabla g)$, and the eigenscheme $\Gamma$ of the canonical derivation $\delta_{f,g}$. First of all note that 
$$ \mathrm{deg}(\Gamma)=\mathrm{deg}(\cB)+\mathrm{deg}(\cZ).$$
In fact, more is true: $\Gamma$ is also the union of $\cB$ and $\cZ$. To see this, consider the commutative diagram below:
$$\begin{CD}
  @. \cO_{\p^2}(2-n-m)  @=
      \cO_{\p^2}(2-n-m) @.  @.\\
   @. @VVV @V{\delta_{f,g}}VV @. \\
     0@>>> \begin{array}{c}\cO_{\p^2}(-1)\\ \oplus\\ \cO_{\p^2}(2-n-m)\end{array}  @>(\delta_E,\delta_{f,g})>> 
      \cO_{\p^2}^3 @>>>  \mathcal{J}_{\Gamma}(n+m-1) @>>>0\\
      @. @VVV @V(\nabla f, \nabla g)VV @VVV  \\
     0 @>>>  \cO_{\p^2}(-1)@>(f,g)>> \begin{array}{c}\cO_{\p^2}(n-1)\\ \oplus \\ \cO_{\p^2}(m-1)\end{array} @>>>  \mathcal{J}_{\cB}(n+m-1)  @>>>0\\
      @. @. @VVV @VVV  \\
   @. @. \cO_{\cZ} @= \cO_{\cZ}   
    \end{CD}$$ 
 The claim follows from the rightmost vertical exact sequence.
\end{proof}

\begin{eg} \label{eigen-conics}
The eigenscheme associated to a pencil of conics with a finite base locus is a finite scheme of length $7$.
When the conics meet in four distinct points, it consists in these $4$ base points plus the $3$ singular points of the singular conics of the pencil.
When the conics meet in three points, it consists in the union of the $2$ singular points of the two singular conics of the pencil plus the scheme of length $5$ supported by the base locus (1+1+3) (See Subsection \ref{conics} for more details). When the conics meet in a quadruple point, the eigenscheme is a finite scheme of length $3$; indeed it is the eigenscheme of the canonical derivation associated to a smooth conic of the pencil and to the reduced line (tangent at the quadruple point) appearing as a double line in the pencil.
\end{eg}

We are now ready to state the main result of this section.

\begin{thm}\label{main-thm}
Let $f\in R_n$ and $g\in R_m$ be two reduced polynomials without common factors such that $V(\nabla f\wedge \nabla g)$  is a finite scheme. Let $\Gamma$ be the eigenscheme associated to the canonical derivation $ \delta_{f,g}$,
let $V(F_k)$ be the union of $k\ge 2$ curves in the pencil generated 
by $(f^a,g^b)$ where $\mathrm{lcm}(n,m)=a\times n=b\times m$ and 
let $F$ be a polynomial of degree $N>n+m-1$ verifying $F \,|\, F_k$. Then $V(F)$
is free with exponents $(n+m-2, N-n-m+1)$ if and only if $F\in 
(I_{\Gamma})_{N}$.
\end{thm}
\begin{proof} This is a direct consequence of Theorem \ref{de-de-new}. Indeed, assume that $F\in 
(I_{\Gamma})_{N}$. Then by Lemma \ref{contain-gamma1} there exists $(Q_1,Q_2,Q_3)\in R_{N-(n+m-1)}^3$ such that 
$$\mathrm{det} 
\begin{pmatrix}  x & P_1  &Q_1  \\
y & P_2  &Q_2 \\
  z & P_3  &Q_3
   \end{pmatrix}= cF, \,\, c\in \kk^*.$$
   As proved in Theorem \ref{de-de-new}, the derivation $Q_1\partial_x+Q_2\partial_y+Q_3\partial_z$ belongs to  $\mathrm{Der}(F)$. Then we conclude by Saito's criterion.
   
\smallskip

Conversely, let us assume that $V(F)$ is free with the given exponents.
The module $\mathrm{Der}_0(F) $ is generated by two derivations $\delta_1$ and $\delta_2$ of degree $n+m-2$ and $N-n-m+1$ by hypothesis. Then $\delta_{f,g}=\delta_1$ if $n+m-2<N-n-m+1$ or $\delta_{f,g}=\delta_1+ h\delta_2$ for some form $h$ of degree $2m+2n-N-3$, if $n+m-2\ge N-n-m+1$. In both cases this gives 
$$ \mathrm{det} (\delta_E, \delta_{f,g},\delta_2)=\mathrm{det} (\delta_E, \delta_1,\delta_2)=cF, \,\, c\in \kk^*.$$
Denoting by  $\delta_2=Q_1\partial_x+Q_2\partial_y+Q_3\partial_z$ and by 
 $\m_i$  the $2\times 2$ minors of the matrix 
 \[
 \begin{pmatrix}  x & P_1    \\
y & P_2  \\
  z & P_3  
   \end{pmatrix}
   \]
this gives $Q_1\m_1+Q_2\m_2+Q_3\m_3=F$ where 
$F\in (I_{\Gamma})_{N}$ as wanted.
\end{proof}

\section{Examples}
\noindent This section is devoted to examples which illustrate the above theorems.

\subsection{First example} Consider the Hesse pencil of a smooth cubic and its Hessian cubic. They meet in $9$ distinct points and the singular curves of the pencil are four triangles. The canonical derivation $\delta_{f,g}$ has degree 4. The eigenscheme is a smooth set of $1+4+4^2=21$ points. This is the union of the 9 base points of the pencil and the 12 vertices of the triangles. Then the union of these four triangles is free with exponents $(4,7).$

\subsection{Second example.} \label{sextic-pencil} As a second example we consider the pencil of sextic curves  $(f^3,g^2)$  where $f(x,y,z)=x^2+y^2+z^2=0$ is a smooth conic and  $g(x,y,z)=xyz=0$ is a triangle
meeting in six different points $A=\{(1,i,0), (1,-i,0), (1,0,i),(1,0,-i),(0,1,i),(0,1,-i)\}$. The locus $V(\nabla f \wedge \nabla g)$ of the "singular points"  of the pencil has length $7$; it consists in  the three vertices of the triangle $(1,0,0), (0,1,0)$ and $(0,0,1)$;  and the four singular points of $f^3-27g^2=0$. These points in the ideal $(x(y^2-z^2),y(z^2-x^2),z(x^2-y^2))$ are $(1,1,1)$, $(-1,1,1)$, $(1,-1,1)$ and $(1,1,-1)$.
We consider the equation 
 $\nabla (\lambda f^3+\mu g^2)(p)=3\lambda f(p)^{2}\nabla f(p)+
2\mu g(p)\nabla g(p)=0,$ 
 and we evaluate at each of the four points. We find the same $\lambda=1$ and $\mu=-27, $ that is one curve $f^3-27g^2=0$ with four singular points.
 
 The canonical derivation $\delta_{f,g}$ has degree 3 and 
the eigenscheme associated to it is a smooth set of 13 points; it is the union of these two sets of simple points $A$ and $V(\nabla f \wedge \nabla g)$. The curve $xyz(f^3-27g^2)=0$, containing the eigenscheme, is free with exponents $(3,5)$ by Theorem \ref{de-de-new}. Moreover, since the curve 
$yz(f^3-g27^2)=0$ (or $xz(f^3-27g^2)=0$ or $xy(f^3-27g^2)=0$) still contains the eigenscheme $\Gamma_{f,g}$  it is also free with exponents $(3,4)$ (Apply again Theorem \ref{de-de-new}).

\subsection{Third example.} \label{conics} We consider a pencil of osculating conics. Up  to a linear transformation, these conics can be defined by $f:\, xz=0$ and $g:\, z^2-xy=0$. The canonical derivation $\delta_{fg}$ has degree $2$ and the associated eigenscheme $\Gamma$ has length $7$ and consists of one smooth point (the intersection point where there is no tangency) and a subscheme of length 6 supported at the point of tangency.
\begin{center}
\definecolor{xdxdff}{rgb}{0.49,0.49,1}
\definecolor{fftttt}{rgb}{1,0.2,0.2}
\begin{tikzpicture}[line cap=round,line join=round,>=triangle 45,x=1.0cm,y=1.0cm]
\clip(-2,-2) rectangle (5.5,2);
\draw [rotate around={5.17:(1.94,-0.13)}] (1.94,-0.13) ellipse (2.84cm and 1.26cm);
\draw [domain=-2:5.5] plot(\x,{(--69.69--62.06*\x)/106.07});
\draw [domain=-2:5.5] plot(\x,{(--0.37-1.94*\x)/1.4});
\begin{scriptsize}
\fill (-0.2,0.54) circle (1.5pt);
\fill (1.2,-1.4) circle (1.5pt);
\end{scriptsize}
\end{tikzpicture}
\end{center}
The ideal defining the eigenscheme is  
$$I_{\Gamma}=\langle x(z^2+xy),x^2z,z^3 \rangle, \mbox{ which we will write as } (u,v,w).$$ 
\begin{itemize}
    \item The equation of the curve 
$fg=xz^3-x^2yz=xw-yv$ belongs to $I_{\Gamma}$ proving that $fg=0$ is free with exponents $(1,2).$ 
\vskip .05in
\item The union of 3 smooth curves of the pencil is also free with exponents $(2,3).$ Indeed, without loss of generality we can choose 3 points in a pencil, corresponding to: 
$g=0$, $f+g=0$ and $f-g=0.$ Then $$g(f+g)(f-g)=w^2-v^2-yz(x-4y+3z)v-xy^2u \in I_{\Gamma}, $$
proving that $g(f+g)(f-g)=0$ is free with exponents $(2,3).$
We note that two smooth osculating curves are not free: they meet in degree 4 along the singular point instead of degree 6. Adding a third smooth curve allows us to reach degree 6. A last remark about this case: at the singular point $p$, the Tjurina number is $15$ and the Milnor number is 16; this shows that $p$ is not a quasihomogeneous singularity.
\vskip -.2in
\begin{figure}[h]
\begin{center}
\includegraphics[width=7in]{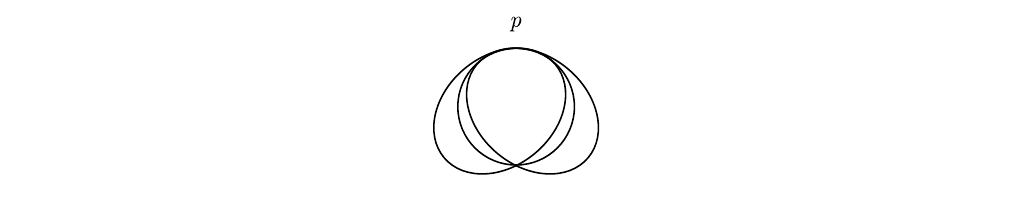}
\end{center}
\vskip -.2in
\end{figure}

\item The union $f(f+g)(f-g)=0$ is also free with exponents $(2,3)$. After removing the smooth conic $f+g$ it remains free, and after removing the transverse line of $f$, we have that $x(f+g)(f-g)=0$ remains free. In this last case we again have a non-quasihomogeneous singularity at $p$: $11=\tau_p\neq \mu_p=12$.
\begin{figure}[h]
\begin{center}
\includegraphics[width=7in]{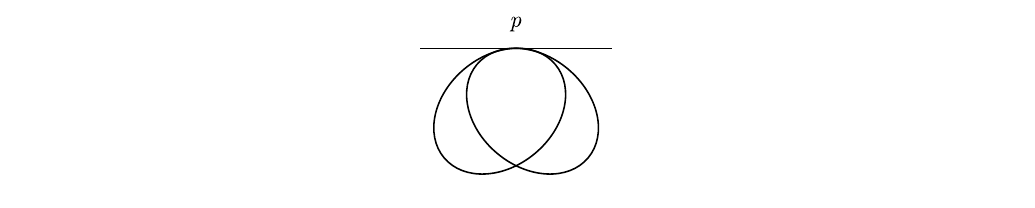}
\end{center}
\vskip -.2in
\end{figure}
\end{itemize}
By Theorem~\ref{main-thm}, a union of conics and lines coming from the pencil will be free if and only if it contains the eigenscheme. So when such a union is free, adding a conic or a line from the pencil to this union remains free. More generally, when a union of curves from a pencil is free, it will remain free by adding smooth curves from the pencil:

\begin{prop} Let $f\in R_n$ and $g\in R_m$ be two reduced polynomials without common factors such that $V(\nabla f\wedge \nabla g)$  is a finite scheme. 
Let $V(F_k)$ be the union of $k\ge 2$ curves in the pencil generated 
by $(f^a,g^b)$ where $\mathrm{lcm}(n,m)=a\times n=b\times m$ and 
let $F$ be a polynomial of degree $N>n+m-1$ such that $F \,|\, F_k$. 

Assume that $V(F)$
is free with exponents $(n+m-2, N-n-m+1)$, and that $C_{\alpha,\beta}=\{\alpha f^a+\beta g^b=0\}$ is  smooth outside the base locus $\cB=V(f)\cap V(g)$.  Then, $V(F)\cup V(C_{\alpha,\beta})$
is free with exponents 
$(n+m-2, N+an-n-m+1)$.
\end{prop}
\begin{proof}
Follows from Theorem \ref{main-thm}: $F\cup C_{\alpha,\beta} \in 
\HH^0(\mathcal{J}_{\Gamma}(N+an))$ since $F\in 
\HH^0(\mathcal{J}_{\Gamma}(N))$.
\end{proof}
As a consequence, it follows that the following types of unions of smooth conics $C,D$ are free: First, if $C\cap D$ consists of $2$ points, one simple and one triple point, then three smooth members of the pencil are needed to contain the eigenscheme. Such a union is free with exponents $(2,3)$. Second, if $C\cap D$ meets in a quadruple point, the union of two smooth members of the pencil contains the eigenscheme (of length $3$) and so is free with exponents $(1,2)$.

\section{Reflection arrangements and nets}
In this section we investigate an example where we add curves coming from a {\em net}, rather than a pencil. The net $(f,g,h)$ we study is defined by the radical ideal of $Jac(F)$, where $V(F)$ is a complete reflection arrangement.
We start by fixing the complete reflection arrangement $F=xyz(x^n-y^n)(x^n-z^n)(y^n-z^n)$; the arrangement $V(F)$ is free with exponents $(n+1,2n+1)$. Computing the partial derivatives we have
$$ \left \lbrace 
    \begin{array}{lcl}
    \partial_x F &= & \frac{F}{x}+\frac{nx^{n-1}F}{x^n-y^n}+\frac{nx^{n-1}F}{x^n-z^n}\\
   \partial_y F&= & \frac{F}{y}-\frac{ny^{n-1}F}{x^n-y^n}+\frac{ny^{n-1}F}{y^n-z^n}\\
    \partial_z F&= & \frac{F}{z}-\frac{nz^{n-1}F}{x^n-z^n}-\frac{nz^{n-1}F}{y^n-z^n}
     \end{array}
       \right.$$
There is a natural derivation of degree $2n+1$, obtained by removing the denominators, which is: 
$$\delta =x(x^n-y^n)(x^n-z^n)\partial_x+y(x^n-y^n)(y^n-z^n)\partial_y+
z(x^n-z^n)(y^n-z^n)\partial_z.$$

\smallskip

\noindent A derivation tangent to $V(F)$ of degree $n+1$ is given in the following lemma of \cite{OT}.

\begin{lem} Let $\mu =x^{n+1}\partial_x+y^{n+1}\partial_y+z^{n+1}\partial_z$ be a derivation. Then, 
$$ \mu \in \mathrm{Der}(x)\cap \mathrm{Der}(y)\cap \mathrm{Der}(z)\cap \mathrm{Der}(x^n-y^n)\cap \mathrm{Der}(x^n-z^n)\cap \mathrm{Der}(y^n-z^n).$$
\end{lem}

\begin{proof}
As $\mu(x)=x^{n+1},$ $\mu \in \mathrm{Der}(x)$, and as $\mu(x^{n}-y^n)=n(x^{2n}-y^{2n})=n(x^n+y^n)(x^n-y^n),$ we have $\mu \in \mathrm{Der}(x^{n}-y^n).$ This also occurs for $y$, $z$, $(y^n-z^n)$ and $(x^n-z^n).$
\end{proof}
\noindent A quick computation shows that 
$$\mathrm{det}(\delta_E,\mu, \delta)=F,$$
which by Saito's criterion proves that $F$ is free. However, we could also prove freeness of $F$ by computing the eigenscheme
$\Gamma_{\mu}$ associated to $\mu$. It is defined by 
the maximal minors  of the matrix $(\delta_E,\mu)$. The defining ideal is  
   $$I_{\Gamma_{\mu}}=(yz(y^n-z^n), xz(z^n-x^n), xy(x^n-y^n)).$$
The length of $\Gamma$ is $(n+1)^2+(n+1)+1=n^2+3n+3$; $\Gamma$ consists of the set of singular points of $F=0$, all with multiplicity one.
Hence it coincides (set theoretically) with the support of the scheme defined by the radical ideal of $Jac(F)$. The curve $F=0$ contains this scheme $\Gamma_{\mu}$.

\begin{thm}
Let $F=xyz(x^n-y^n)(x^n-z^n)(y^n-z^n)=0$ be the complete reflection arrangement  of $3n+3$ lines. 
Let $G_i$  be a general polynomial in the net $(I_{\Gamma_{\mu}})_{n+2}$.  Then 
\begin{enumerate}
\item $FG_i=0$ is free with exponents $(2n+2,2n+2).$
\item $FG_1G_2=0$ is free with exponents $(2n+2,3n+4).$
\item $F\prod_{1\le i\le k} (a_iG_1+b_iG_2)=0$ is free with exponents $(2n+2,(k+1)n+2k).$
\item $FG_1G_2G_3=0$ is free with exponents $(3n+4,3n+4).$
\end{enumerate}
\end{thm}
 \begin{proof}
(1) The singular points of the net 
$$(I_{\Gamma_{\mu}})_{n+2}=\{ yz(y^n-z^n), xz(z^n-x^n), xy(x^n-y^n)\}=\{ f,g,h \}$$
define a curve in $\p^2$ with equation 
$$ \mathrm{det}(\nabla f, \nabla g,\nabla h)=0.$$
Computing this determinant we obtain
 $$ \mathrm{det}(\nabla f, \nabla g,\nabla h)= n(n+1) F.$$
We now consider the following three derivations of degree $(2n+2)$: 
 $$\delta_{fg}=\mathrm{det}(\nabla f, \nabla g,\nabla ), $$
  $$\delta_{fh}=\mathrm{det}(\nabla f, \nabla h,\nabla ), $$
   $$\delta_{gh}=\mathrm{det}(\nabla g, \nabla h,\nabla ). $$
   Given a curve $G_1=af+bg+ch=0$ in the net $(f,g,h)$,  we have 
  $$\delta_{fg}( af+bg+ch)= n(n+1)cF$$ 
  $$\delta_{fh}( af+bg+ch)=-n(n+1)bF, $$
   $$\delta_{gh}( af+bg+ch)= n(n+1)aF.$$
 This gives a pencil of derivations of degree $2n+2$  
 $$(\delta_1,\delta_2)=(b\delta_{fg}+c\delta_{fh},a\delta_{fh}+b\delta_{gh}) $$ 
 belonging to $\mathrm{Der}_0(af+bg+ch)=\mathrm{Der}_0(G_1).$
 
 \smallskip
 
 We want to prove that these derivations belong to $\mathrm{Der}((af+bg+ch)F)=\mathrm{Der}(FG_1)$. Since $\delta_i((af+bg+ch)F)=(af+bg+ch)\delta_i(F)$ it remains to prove that $\delta_i(F)\subset (F).$
 
 \medskip

  The relation $\delta_1(f)=0$ implies $\delta_1(yz)\subset (yz)$. Moreover $\delta_1(af+bg+ch)=0$ gives  $\delta_1(bg+ch)=0$ implying  $\delta_1(x)\subset (x).$ Finally $\delta_1(xyz)\subset (xyz)$. We have also $\delta_1(g)=cF\subset (g)$ and $\delta_1(h)=bF\subset (h)$, proving that $\delta_1(fgh)\subset (fgh).$ Since $fgh=xyzF$, the inclusions
 $\delta_1(xyz)\subset (xyz)$ and $\delta_1(fgh)\subset (fgh)$ imply $\delta_1(F)\subset (F)$.

When $G_1$ is general both derivations are non-proportional and since $\mathrm{deg}(FG_1)=4n+5$ this allows us to conclude that $FG_1=0$ is free with exponents $(2n+2,2n+2).$ 

\medskip

To prove items (2) and (3), note that after adding a new curve $G_2=a'f+b'g+c'h=0$ from  the net, we have only one derivation of degree $2n+2$ which is 
 $$\nu= \nabla(af+bg+ch)\wedge \nabla (a'f+b'g+c'h)=(ab'-a'b)\delta_{fg}+(ac'-a'c)\delta_{fh}+(bc'-b'c)\delta_{gh}.$$
This is the canonical derivation associated to the pencil of degree $n+2$ curves $(G_1,G_2)$.
 The derivation $\nu$ belongs to $\Der{FG_1}$. Since $FG_1=0$ is free with exponents $(2n+2,2n+2)$ this means that this curve contains the eigenscheme $\Gamma_{\nu}$ associated to $\nu$ by Theorem \ref{de-de-new}. The curve $FG_1G_2=0$ contains a fortiori this eigenscheme and $\nu \in \mathrm{Der}(FG_1G_2)$. By Theorem \ref{de-de-new} this proves the second and third assertions.
 
 \medskip
 
 (4) We denote by $\alpha$ a derivation of degree $3n+4$ such that $$\mathrm{det}(\delta_E,\nu,\alpha)=FG_1G_2.$$
 There is a natural derivation of degree $3n+4$ in $\mathrm{Der}(FG_1G_2G_3)$ which is 
 $G_3\nu.$ Note that 
 $$\mathrm{det}(\delta_E,G_3\nu,\alpha)=FG_1G_2G_3.$$
 Taking linear combinations allows us to transform the derivation  $G_3\nu$ to a irreducible derivation 
 $\nu_1\in \mathrm{Der}_0(FG_1G_2G_3)$ such that 
 $$\mathrm{det}(\delta_E,\nu_1,\alpha)=FG_1G_2G_3.$$
 By Lemma~\ref{contain-gamma1} this shows that $FG_1G_2G_3=0$ contains the eigenscheme associated to the derivation $\nu_1.$ Now since  $\nu_1\in \mathrm{Der}_0(FG_1G_2G_3)$, and we are done by Theorem \ref{de-de-new}.
 \end{proof}
 %
  %
  %
  %
%
\begin{remark}
Adding a fourth general $G_i$ from the net yields the polynomial $FG_1G_2G_3G_4$, which is free with exponents $(3n+5,4n+5).$ On the other hand, adding a fifth general $G_i$ yields $FG_1G_2G_3G_4G_5$, which is not free. We are investigating the behavior of freeness (in general) when adding elements of a net. 
\end{remark}

\end{document}